\documentclass[12pt]{article}
\usepackage{amssymb,amsmath,amsthm,amscd,amsfonts}

\newtheorem{theorem}{Theorem}[section]
\newtheorem{lemma}[theorem]{Lemma}
\newtheorem{proposition}[theorem]{Proposition}
\newtheorem{corollary}[theorem]{Corollary}
\theoremstyle{definition}
\newtheorem{definition}[theorem]{Definition}
\newtheorem{example}[theorem]{Example}

\newtheorem{remark}[theorem]{Remark}
\numberwithin{equation}{section}

\date{}
\title{Nakayama's Lemma on $\textbf{Act}-S$}

\author{Kamal Ahmadi\\ Ali Madanshekaf\\  Department of Mathematics,\\ Semnan  University,\\ P. O. Box 35131-19111,\\ Semnan, Iran.\\
emails:kamal.ahmadi.math@gmail.com\\ amadanshekaf@semnan.ac.ir}

\begin{document}
\maketitle
\begin{abstract}
A crucial lemma on module theory  is Nakayama's lemma~\cite{AF}. In this article,  we shall investigate some forms of Nakayama's lemma in the category of right acts over a given monoid $S$ with identity 1. More precisely, among other things, we show that equality $AI=A$ for some proper ideal $I$ of $S$ implies $A=\{\theta\}$, when $A$ is a finitely generated quasi-strongly faithful $S$-act with unique zero element $\theta$ and $S$ is a monoid in which its
unique maximal right ideal $\mathfrak{M}$ is two-sided. Furthermore, as an application of Nakayama's lemma we prove Krull intersection theorem for $S$-acts. Finally, as a consequence, we shall see a homological classification form of this  lemma, i.e, we prove if $S$ is a commutative monoid then every projective $S$-act is free if and only if $E(S)=\{1\}$, which $E(S)$ is the set of all idempotents of $S$.
\end{abstract}
AMS {\it subject classification}: 20M30, 20M50. 

\section{Introduction and Preliminaries}

Nakayama's Lemma was first discovered in the special case of ideals in a commutative ring by Wolfgang Krull and then in general by Goro
 Azumaya~\cite{GA}. The lemma is named after the Japanese mathematician Tadashi Nakayama and introduced in its present form in Nakayama~\cite{TN}.
 It is a significant tool in algebraic geometry, because it allows local data on algebraic varieties, in the form of modules over local rings,
to be studied pointwise as vector spaces over the residue field of the ring.

  There are several equivalent forms of Nakayama's Lemma in algebra. We express one here. Let $R$ be a ring with identity 1, and $A$ a finitely
  generated right $R$-module. If   $I$  is a right ideal of $R$ contained in the Jacobson radical of $R, J(R),$ and $AI=A$  then $A=0$.

  Some generalizations of Nakayama's Lemma  has been given and studied, in the literatures. For example, R. Ameri
in~\cite{AM}  generalized Nakayama's Lemma to a class of multiplication modules over commutative rings with
identity. Also, A. Azizi~\cite{AZ}  introduced Nakayama property for modules over a
 commutative ring with identity. He says that an  $R$-module $M$ has Nakayamaya property if $IM = M$,
where $I$ is an ideal of $R$, implies that there exists $a\in R$
such that $aM = 0$ and $a-1\in I.$ Then Nakayama's Lemma  states that every finitely generated $R$-module has Nakayama property.
He proved that $R$ is a perfect ring if and only if every $R$-module has Nakayama property.
Besides, we inform that there are generalizations in other contexts, we refer the reader to ~\cite{p.b} and \cite{ogus}.

Throughout this paper, $S$ is a monoid with at least a right non-invertible element, all $S$-acts will be right $S$-acts and all
ideals of $S$ are right ideals. The set of all idempotents of $S$ is denoted by $E(S)$. It is known that the set
$$\{ s\hspace{1mm}\mid \ s  \textrm{ is a right non-invertible element of }  S \}$$
  is the only maximal right ideal of $S$. In this note, we reserve  $\mathfrak{M}$ to denote, always, this unique maximal right ideal of $S$.
For more information on  $S$-acts we refer the reader to~\cite{KKM}.

 Let $S$ be a monoid with identity $1$. Recall that a (right) $S$-act is a nonempty set $A$ equipped with a map
$\mu: A\times S\to A$ called its action, such that, denoting $\mu(a,s)$ by $as$, we have $a1 = a$ and $a(st)=(as)t$, for all $a\in A$,
and $s, t\in S$. An element $\theta \in A$  is called a zero
of $A$ if $\theta s=\theta$ for any $s\in S$. Let $A$ be an $S$-act and $B\subseteq A$ a non-empty subset. Then
$B$ is called a subact of $A$ if $bs\in B$ for all $s\in S$ and $b\in B.$ In particular, if $I$ is a (proper) ideal of $S$, then $AI:= \{
as \mid a\in A, s\in I\}$ is a subact of $A.$ Any subact $B\subseteq
A$ defines the Rees congruence $\rho_B$ on A, by setting $a\rho_B a'$ if $a, a'\in B$ or $a = a'$. We denote the resulting factor act
by $A/B$ and call it the Rees factor act of $A$ by the subact $B$.
Clearly, $A/B$ has a zero which is the class consisting of $B$, all other
classes are one-element sets. Moreover, any subact $B\subseteq A$ gives rise to a kernel congruence $\ker \pi$ where $\pi : A\rightarrow A/B$
 is the canonical epimorphism. The category of all $S$-acts, with
action-preserving ($S$-act) maps ($f : A\to B$ with $f(as) = f(a)s$, for $s\in S$, $a\in A$) between them, is denoted by
$\textbf{Act}-S$. Clearly $S$ itself is an $S$-act with its operation as the action.

  Nakayama's Lemma for $R$-modules governs the interaction between the Jacobson radical of a ring and its
finitely generated modules. We aware that for a ring  $R$ with identity 1, $J(R)$ is a two-sided ideal of $R$, but for a monoid $S$,
$$\mathfrak{M}=\{s\in S \ \mid \ st\neq 1 \textrm{ for all }  t\in S\}$$ is the only maximal right ideal of $S$ and
then for any proper ideal $I$ of $S$ we have $I\subseteq \mathfrak{M}$. We can not talk about Jacobson
radical of a monoid, because $\mathfrak{M}$ is the only maximal ideal for it. We therefore  consider Nakayama's
Lemma in $\textbf{Act}-S$ where $S$ is a monoid with a unique two-sided maximal ideal $\mathfrak{M}$.
In~\cite[Example 3.18.10]{KKM}, the authors present a monoid $S$ in which $\mathfrak{M}$ is not a two-sided
ideal. Moreover, there are examples of monoids $S$ that are not commutative, but their maximal ideals are two-sided. For example, given
$S=(M_n(\mathbb{R}),\cdot )$, the monoid of all $n\times n$ matrices with real number entries  under usual multiplication of matrices.
Since $ab=1$ implies $ba=1$ for $a, b\in S$, the unique maximal ideal of $S$ is two-sided (see Lemma~\ref{two-sided equivalent
condition} of this paper). Besides, there are many examples of finitely
generated $S$-acts $A$ with zero element $\theta$ in which for a proper ideal $I$ of $S$, $AI=A$, but $A\neq\{\theta\};$ take any
monoid $S$ and an arbitrary finite set $A$ with $|A|>1$. Then $A$ becomes a right $S$-act by trivial action, i.e., $as=a$ for all $a\in A, s\in
S.$ Therefore, $AI=A$ for every proper ideal $I$, although, $A\neq\{\theta\}$. As another  example of this situation,  we
will provide  in Example~\ref{residue class group modulo} of next section,  a finitely generated $\mathbb{N}$-act with a unique zero for which $A(2\mathbb{N})=A$, but $A\neq\{\theta\}$.

In this paper, we are going to study some forms of Nakayama's Lemma on $\textbf{Act}-S$.
More precisely, in Theorem~\ref{nakayama 2} of next section, we show that equality $AI=A$ for some proper ideal $I$ of $S$ implies $A=\{\theta\}$,
when $A$ is a finitely generated quasi-strongly faithful $S$-act with unique zero element $\theta$ and $S$ is a monoid in which the
unique maximal right ideal $\mathfrak{M}$ is two-sided.  As an application of Nakayama's Lemma we prove Krull intersection
theorem for $S$-acts. As a consequence, we shall see a homological classification form of this  lemma, i.e, we prove if $S$ is a
commutative monoid then, every projective $S$-act is free if and only if $E(S)=\{1\}$, which $E(S)$ is the set of all idempotents of $S$.

\section{Some forms of Nakayama's lemma}

In this section we will prove  that a version of Nakayama's Lemma for $\textbf{Act}-S$.
We begin with the following lemma which determines monoids in which their unique maximal ideals are two-sided.
\begin{lemma}\label{two-sided equivalent condition}
Let $S$ be a monoid. Then the following statements are equivalent:
\begin{enumerate}
    \item $\mathfrak{M}$ is a two-sided ideal of $S$.
    \item $st=1$ implies $ts=1$, for all $s,t\in S$.
\end{enumerate}
\begin{proof}
$(1)\Rightarrow (2).$ Suppose  $\mathfrak{M}$ is two-sided and
$st=1$ and $ts\neq 1$ for some $s,t\in S$. Then $tl\neq 1$ for all
$l\in S.$ Otherwise, $tl=1$ for some $l\in S$ implies that
$(st)l=s$, and so $l=s$. Hence $ts=1$ which is a contradiction. Then
$tl\neq 1$ for all $l\in S.$ Now according to the definition of $\mathfrak{M},
t\in\mathfrak{M}.$ But $\mathfrak{M}$ is a two-sided ideal and
$st=1,$  so $1=st\in\mathfrak{M}$ which is a contradiction. Thus
$st=1$ implies that $ts=1$.\\
$(2)\Rightarrow (1)$. Suppose $t\in\mathfrak{M}$ and $s\in S$ are
such that $st\notin\mathfrak{M}.$ Then, $s(ty)=(st)y=1$ for
some $y\in S$. By assumption we have $t(ys)=(ty)s=1$. Therefore,
 $t\notin\mathfrak{M},$ which is  a contradiction.
\end{proof}
\end{lemma}
The following remark is a simple consequence of Lemma~\ref{two-sided equivalent condition}.
\begin{remark} Recall from~\cite{KKM} that  monoids $S$ and
$T$ are  Morita equivalent  if $\textbf{Act}-S$ and $\textbf{Act}-T$
are equivalent categories. In this case, we write $S\sim_M T$. In a
certain sense Morita equivalence is the global aspect of homological
classification of monoids since it contains the question to which
extent a monoid $S$ is determined by the entire category
$\textbf{Act}-S$. For monoids $S$ and $T,$ in view of~\cite[Corollary
5.3.14]{KKM} and Lemma~\ref{two-sided equivalent condition} we
obtain that, $S\sim_M T$ if and only if $S\cong T$, whenever the unique maximal ideals of $S$ and $T$ are two-sided.
\end{remark}
The next lemma easily shows that every finitely generated $S$-act with a unique zero element has a maximal subact.
 \begin{lemma}\label{maximal subact}
Let $S$ be a monoid and let $A$ be a finitely generated $S$-act with
 a unique zero element $\theta$ and $A\neq\{\theta\}$. Then every proper
subact of $A$ is contained in a maximal subact. In particular, $A$
has a maximal subact.
\end{lemma}
\begin{proof}
Given $B$  a proper subact of $A$. Let  $\sum$ be the set of all
proper subacts of $A$ that contain $B.$ Then $\sum$ is a partially ordered set with set inclusion order.
Now, a straightforward application of  Zorn's lemma gives us the result.
\end{proof}
In the sequel we will present two versions of Nakayama's lemma.
\begin{theorem}({\bf Nakayama's Lemma, general case})\label{nakayama 1}
Let S be a monoid in which its unique   maximal right ideal
$\mathfrak{M}$ is two-sided. Moreover, let $A$  be an $S$-act and $B$  a
maximal subact of $A$ in which there exists $a\in A\setminus B$ such
that $\mathfrak{M} = \{s\in S \ \mid \ as\in B\}$. Then for any
proper ideal $I$ of $S$ we have $AI\neq A.$
\begin{proof}
On the contrary, let $I$ be a proper ideal of $S$ such that
$AI=A$. Then there are elements $c\in A$ and $s\in I$ such that
$a=cs$. It is clear that $c\neq a$, because if $c=a$ then  $a=as$
and  since $s\in I\subseteq \mathfrak{M}$ we get $a\in B$ which is a
contradiction. Since $a\notin B$ we have $c\notin B$ and since $B$
is a maximal subact of $A$ we have $A=B\cup aS$. Then, $c\in aS$,
and so $c=at$ for some  $t\in S.$ Since $s\in \mathfrak{M}$ and
$\mathfrak{M}$ is a  two-sided ideal of $S$ we have $ts\in
\mathfrak{M}.$ On the other hand,  $a=cs=ats\in B$ and this is a
contradiction. Therefore, for any proper ideal $I$ of $S$ we have $AI\neq A$.
\end{proof}
\end{theorem}
\begin{example}\label{first demonstration}
Let $(\mathbb{N},\cdot )$ be the monoid of natural numbers with
the usual multiplication. Then  it is an $\mathbb{N}$-act,
$2\mathbb{N}$ is a subact of $\mathbb{N}$ and it satisfies in the
assumptions of Theorem~\ref{nakayama 1}. So for each proper  ideal  $I$ of
$\mathbb{N}$ we have $(2\mathbb{N})I\neq 2\mathbb{N}$.
\end{example}
\begin{example}\label{second demonstration}
Let $S=(\mathbb{N},\cdot )$ be the monoid of natural numbers with
the usual  multiplication. Then, again $A=S\backslash\{1\}$ with
the usual multiplication of natural numbers as its operation is an $S$-act.
Let $U$ be the set of all prime numbers. Then $U$ is a set of
generating elements of $A$. Note that $U$ is the least generating set of $A$, i.e. $A$ is not finitely generated. The set
$B=S\backslash\{1,2\}$ is a maximal subact of $A$ in which $\mathfrak{M}=\{s\in S \mid 2s\in B\}.$ Then all  conditions of
Theorem~\ref{nakayama 1} hold for  $A$. Therefore for any proper ideal $I$ of $S$, $AI\neq A$.
\end{example}
In the next lemma we will see that the second condition of
Theorem~\ref{nakayama 1} is equivalence to the implication that $as=a$ implies $s\notin\mathfrak{M}$. More precisely:
\begin{lemma}\label{two equivalent condition}
Let $A$ be an $S$-act with a maximal subact $B$. For any $a\in A\setminus B$ the following statements are equivalent:
\begin{enumerate}
\item $\mathfrak{M}=\{s\in S \mid as\in B\}$.
\item $as=a$ implies $s\notin\mathfrak{M}$.
\end{enumerate}
\end{lemma}

\begin{proof}
$(1)\Rightarrow (2).$ If $as=a$ and $s\in\mathfrak{M}$, then $a\in B$ which is impossible.\\
$(2)\Rightarrow (1)$. Since $a\notin B,$ the set $ \{s\in S  \mid  as\in B\}$
is a proper ideal of $S$. On the other hand, by (2)
 we have $a\notin a\mathfrak{M}$, then  $B\subseteq B \cup a\mathfrak{M}\subsetneqq A$.
 As $B$ is a maximal subact  of $A$, $B=B\cup a\mathfrak{M}$. This means that $a\mathfrak{M}\subseteq B$, or
$\mathfrak{M}\subseteq\{s\in S
 \mid  as\in B\}\subsetneqq S$, but $\mathfrak{M}$ is the  maximal ideal of $S$,
therefore $\mathfrak{M}=\{s\in S  \mid  as\in B\}$.
\end{proof}
If $A$ is a right $R$-module on local ring $R$ then
$J(R)=\mathfrak{M}$ is the only maximal right ideal of $R$. In this case we have, $a=am$ if and only if $a=0$, for all
$m\in\mathfrak{M} $ and $a\in A$. Therefore the following corollary is a version of Nakayama's lemma for $S$-acts.

\begin{corollary}\label{nakayama 1.5}
Let S be a monoid for which  the  maximal ideal $\mathfrak{M} $ is
two-sided and $A$ is an $S$-act with a maximal subact $B$. Let
$AI=A$, for some proper ideal $I$ of $S$. Then for every $a\in
A\backslash B$ there exists $m\in\mathfrak{M} $ such that $am=a$.
\end{corollary}

\begin{proof}
If for some $a\in A\backslash B,$  $am\neq a$, for every
$m\in\mathfrak{M} ,$ then by Lemma~\ref{two equivalent condition},
$\mathfrak{M} =\{s\in S  \mid  as\in B\}.$  By Theorem~\ref{nakayama 1},
$AI\neq A$, which  contradicts to the assumption. The proof  is now
complete.
\end{proof}
\begin{example}\label{residue class group modulo}
Let $A=\mathbb{Z}_r=\{[0],[1],\cdots ,[r-1]\}$ be the additive group
of integers modulo $r$ where $r>1$ and let $(\mathbb{N},\cdot )$ be
the monoid of natural numbers with the usual multiplication. Then, $A$
is an $\mathbb{N}$-act by the action $[t]n=[tn]$. By definition, $A$
is a finitely generated $\mathbb{N}$-act with a unique zero.
However,  for the ideal $I=2\mathbb{N}$ of $\mathbb{N}$ we have
$AI=A$ if $A=\mathbb{Z}_3$. Indeed, $B=\{[0]\}$ is a maximal subact
of $\mathbb{Z}_3$, according to Corollary~\ref{nakayama 1.5} for any
$a\in \mathbb{Z}_3$ there exists $m\in\mathfrak{M} ,$ such that
$am=a$.
\end{example}

\begin{definition}
We call an $S$-act $A$  quasi-strongly faithful if for $s\in S$ the
equality $as=a$ for some element $a\in A$  implies that
$s\notin\mathfrak{M} $. (In case $A$ has a unique
 zero element $\theta$ we assume that, $a\neq\theta$.)
\end{definition}

\begin{theorem}({\bf Nakayama's Lemma, particular case})\label{nakayama 2}
Let S be a monoid in which its unique maximal right ideal
$\mathfrak{M} $ is two-sided. Let $A$ be a finitely generated
quasi-strongly faithful $S$-act with a unique zero element $\theta$.
If $AI=A$ for some proper ideal $I$ of $S$, then $A=\{\theta\}$.
\end{theorem}
\begin{proof}
Assume that $A\neq\{\theta\}.$ Then  by Lemma~\ref{maximal
subact}, $A$ contains a maximal subact $B$. Take $a\in A\setminus
B.$ By Corollary~\ref{nakayama 1.5}, $as=a$ for some
$s\in\mathfrak{M} $. This is a contradiction because $A$ is a
quasi-strongly faithful $S$-act. Thus $A=\{\theta\}$.
\end{proof}
\begin{corollary}\label{nakayama 3} Let S be a monoid in which its
unique maximal right ideal $\mathfrak{M} $ is two-sided. Let $A$ be a
finitely generated quasi-strongly faithful $S$-act. If $B\cup AI=A$
for some proper ideal $I$ of $S$ and some subact $B$ of $A$, then
$A=B$.
\end{corollary}
\begin{proof}
Let $X=\{x_1,x_2,\cdots,x_n\}$ is a generating set for $A$. Then the set $\{\pi (x_i)\mid 1\leq i\leq n\}$ is a generating set for the
$S$-act $A/B$ where $\pi : A\rightarrow A/B$ is the canonical epimorphism. Since $B\cup AI=A$ we readily obtain $(A/B)I=A/B.$ On the
other hand, since $A$ is finitely generated and quasi-strongly faithful,  $A/B$ is too. Note that  $A/B$ always is an
 $S$-act with a unique zero element therefore by Theorem~\ref{nakayama 2}, we get $A/B=\{\theta\}$. That is,   $A=B$.
\end{proof}

\section{Application}

In this section we provide some applications of Nakayama's Lemma for $S$-acts.
\begin{proposition}\label{idempotent ideal}
Let S be a monoid in which its unique maximal right ideal $\mathfrak{M} $ is two-sided and
$A$ be an $S$-act which satisfies in the conditions of Theorem~\ref{nakayama 1}. Then:
\begin{enumerate}
    \item For every proper ideal $I$ of $S$ with $I^2=I$ we have $AI\ncong A$.
    \item For every ideal $I$ of $S$, $AI=A$ if and only if $I=S$.
    \item For every ideal $I$ of $S$ with $I^2=I$, $AI\cong A$ if and only if $I=S$.
\end{enumerate}
\begin{proof}
(1) In the contrary,  assume that there exists  a proper ideal $I$ of $S$, such that $I^2=I$ and
 $AI\cong A$. Let  $f:AI\longrightarrow A$ be an isomorphism.
Then $f(AI)=A$. We have $AI = f(AI)I=f(AI^2)=f(AI) = A.$ This is a contradiction with Nakayama's Lemma.\\
(2) $I\neq S$ implies $AI\neq A$ by Nakayama's Lemma. Therefore $AI=A$ implies $I=S$. Since $a1=a$, for any $a\in A$ we have
$AS=A.$ Therefore, $I=S$ implies that $AI=A$.\\
(3) follows from (1).
\end{proof}
\end{proposition}
Next, as an application of Nakayama's Lemma,  we characterize commutative monoids $S$ over which every projective $S$-act is free.
By~\cite[Theorem 3.17.8]{KKM}, an act $P\in\textbf{Act}-S$ is projective if and only if $P=\dot{\cup}_{i\in I}P_i$ where $P_i\cong
e_iS$ for idempotents $e_i\in S, i\in I$. By~\cite[Theorem 4.13.3]{KKM}, if $S$ is a commutative monoid then it is easy to see
that every projective $S$-act is free $S$-act if and only if $E(S)=\{1\}$.

In the next theorem we will prove this result by Nakayama's Lemma.
\begin{theorem}
Let $S$ be a commutative monoid. Then every projective $S$-act is free if and only if $E(S)=\{1\}$.
\end{theorem}

\begin{proof}
First, assume that every projective $S$-act is free and $e\in S$
is an idempotent. Since $S$ is a commutative monoid, its unique
maximal ideal is two-sided and hence $SeS=eS.$ On the other hand,
it is easy to check that  $eS$ is a projective $S$-act (see
also~\cite[Theorem 3.17.2]{KKM}), therefore it is  a free $S$-act.
In view of~\cite[Theorem 1.5.13]{KKM}, $eS\cong\dot{\cup}_{i\in
I}S_i$ where $S_i\cong S$ for any $i\in I.$ By~\cite[Proposition
1.5.8]{KKM}, $eS$ is an indecomposable $S$-act, hence $eS\cong S.$
Now by part (3) of Corollary~\ref{idempotent ideal}, we get, $eS =
SeS \cong S$ if and only if $eS=S$. Then there exists $s\in S$
such that $es=1$ and so $e = es = 1.$ Thus $E(S) = \{1\}$. The converse is clear.
\end{proof}

In the following, we show that if $S$ is a monoid in which its
unique maximal right ideal $\mathfrak{M} $ is two-sided, then every two minimal generating sets of any
finitely generated quasi-strongly faithful $S$-act have the same size.

\begin{lemma}\label{dimention1}
Let $S$ be a monoid with a unique zero element in which every nonzero element of $S$ is invertible and $A$ a finitely generated
$S$-act. If $X$ and $Y$ are any minimal generating sets for $A$ then $|X|=|Y|$.
\end{lemma}
\begin{proof}
Assume that $X=\{x_1,\cdots,x_n\}$ and $Y=\{y_1,\cdots,y_m\}$ are minimal generating sets for $A$ and $n\neq m$. We may assume without
loss of generality that $n<m.$ For any $1\leq i\leq m,$ there exists $x_j\in X( 1\leq j\leq n),$ and  a nonzero element $s\in S$ such
that $y_i=x_js.$ Since $m>n,$  for distinct pair of elements $i_1,i_2\in\{1,2,\cdots,m\}$ there is exactly an index
$j\in\{1,2,\cdots,n\}$ such that  $y_{i_1}=x_js_1$ and $y_{i_2}=x_js_2$. Then we have $x_j=y_{i_1}s_1^{-1}=y_{i_2}s_2^{-1}$
and so  $y_{i_1}=y_{i_2}s_2^{-1}s_1$ that contradicts by the minimality of $Y$. Therefore $n=m$ and $|X|=|Y|$.
\end{proof}
\begin{theorem}\label{dimention2}
Let S be a monoid in which its unique maximal right ideal $\mathfrak{M} $ is two-sided. Let $A$ be a finitely generated
quasi-strongly faithful $S$-act,   $\pi : A\rightarrow A/A\mathfrak{M} $
 the canonical epimorphism and $a_1,\cdots,a_n\in A$. Then the following statements are equivalent:
\begin{enumerate}
    \item A is generated by $a_1,\cdots,a_n$;
    \item the $S$-act $A/A\mathfrak{M} $ is generated by
    $\pi(a_1),\cdots,\pi(a_n)$;
    \item the $S/\mathfrak{M} $-act $A/A\mathfrak{M} $ is generated by $\pi(a_1),\cdots,\pi(a_n)$.
\end{enumerate}
\end{theorem}
\begin{proof}
(1)$\Rightarrow$ (2) is clear.\\
(2)$\Rightarrow$ (1) Let $B$ be the subact of $A$ generated by
$a_1,\cdots,a_n$. We show that $A=B\cup A\mathfrak{M} $. Assume
$a\in A$ then $\pi(a)\in A/A\mathfrak{M} $. By (2) we have
$\pi(a)=\pi(a_i)s$ for some $s\in S$ and $1\leq i\leq n$. If
$\pi(a)=A\mathfrak{M} $ then $a\in A\mathfrak{M} $ and otherwise
$a=a_is$ therefore $a\in B$. Then we have $A=B\cup A\mathfrak{M} $
therefore $A=B$ by Corollary~\ref{nakayama 3}.\\
(2)$\Rightarrow$ (3) is clear.\\
(3)$\Rightarrow$ (2) Assume $A\mathfrak{M} \neq \pi(a)\in
A/A\mathfrak{M} $ then $\pi(a)=\pi(a_i)[s]$ for some $[s]\in
S/\mathfrak{M} $ and $1\leq i\leq n$ by $(3)$. Since
$A\mathfrak{M} \neq \pi(a)$ therefore $s\notin \mathfrak{M} $ and
$[s]=\{s\}$. Then $\pi(a)=\pi(a_i)s$ which  implies $(2)$.
\end{proof}
\begin{corollary}
Let S be a monoid in which the unique maximal right ideal
$\mathfrak{M} $ is two-sided. Let $A$ be a finitely generated
quasi-strongly faithful $S$-act. If $X, Y$ are minimal generating
sets for $A$ then $|X|=|Y|$.
\end{corollary}
\begin{proof}
Since $\mathfrak{M} $ is a two-sided ideal of $S$,
$S/\mathfrak{M} $ is a monoid in which every nonzero element is
invertible. Then the result follows by Lemma~\ref{dimention1}
and Theorem~\ref{dimention2}.
\end{proof}
Let $I$ be an ideal of a commutative Noetherian ring $R$ such that
$I\subseteq J(R)$. Then $\cap_{n\in\mathbb{N}}I^n=0$. This is known as Krull's Intersection Theorem (see~\cite[Corollary 8.25]{sharp}) in
the theory of modules over commutative rings. In the following we prove a counterpart  of this theorem for $S$-acts.

Recall~\cite{chain} that an $S$-act $A$ is  Noetherian if $A$ satisfies
the ascending chain condition on its subacts, that is, every ascending chain of subacts of $A$ is finite.

\begin{lemma}\label{noetherian}
Let S be a monoid and $A,B$ be two $S$-acts.
\begin{enumerate}
    \item $A\dot{\cup} B$ is Noetherian if and only if $A$ and $B$ themselves are;
    \item if $B$ is a subact of $A$, then $A$ is Noetherian if and only if $B$ and $A/B$ themselves are;
    \item if $S$ is a Noetherian monoid and $A$ is a finitely generated $S$-act then $A$ is Noetherian.
\end{enumerate}
\end{lemma}
\begin{proof}
For proofs of (1) and (2) see~\cite[Lemmas 4.1 and 4.2]{chain}. Now we prove (3).
By~\cite[Proposition 1.5.16  and Corollary 1.4.22]{KKM}, $A\cong F/K$ where $F\cong \dot{\cup}_{i\in I}S_i$ and $K$ is a
subact of $F$ and $|I|<\infty$ and $S_i\cong S$ for each $i\in I$. Since $S$ is Noetherian the result follows from (1) and (2).
\end{proof}
\begin{theorem}\label{krull}
Let $S$ be a Noetherian monoid, $I$ an ideal of $S$ and $A$ a finitely generated
$S$-act. If  $B=\cap_{n\in\mathbb{N}}AI^n$, then $BI=B.$
\end{theorem}
\begin{proof}
By part (3) of Lemma~\ref{noetherian}, $A$ is Noetherian. Put
$\sum$ to be the set of all proper subacts $D$ of $A$ such that $D\cap B=BI.$
Then $\sum$ with set inclusion is a partially ordered set and $\sum\neq\emptyset$ since
$BI\in\sum$. By Zorn's lemma, $\sum$ has a maximal element $C.$
If we can show that $AI^n\subseteq C$ for some $n\in\mathbb{N}$,
then $B\subseteq AI^n\subseteq C$ and since $C\in\sum$ we get the result. To this end, let
$s\in I$ and let $$A_k=\{a\in A\mid as^k\in C\}$$ where
$k\in\mathbb{N}$. Since $A$ is Noetherian and $A_1\subseteq
A_2\subseteq\cdots$ there exists some $m\in\mathbb{N}$ such that $A_m=A_{m+n}$
for every $n\in\mathbb{N}$. We claim that
\begin{equation}\label{claim}
(As^m\cup C)\cap B=BI
\end{equation}
If $x\in (As^m\cup C)\cap B$ then $x\in As^m\cap B$ or $x\in C\cap
B$. If $x\in C\cap B$ then $x\in BI$ since $C\in$. If $x\in
As^m\cap B$ then there exists $a\in A$ such that $x=as^m$ and $xs\in BI$ since
$x\in B.$ Therefore, $as^{m+1}\in BI=C\cap B$ and $as^{m+1}\in C$.
Since $as^{m+1}\in C$ and $A_m=A_{m+1}$, $x=as^m\in C$. Then,
$x\in C\cap B=BI$ and so $(As^m\cup C)\cap B\subseteq BI.$ the reverse inclusion is evident. Thus \ref{claim} holds.
Now, if $A=As^m\cup C$ then $BI=B$ and we get the result. If $A\neq As^m\cup C$, as $As^m\cup
C\in \sum,$ so $C=As^m\cup C$ because $C\subseteq As^m\cup C$
and $C$ is a maximal element of $\sum$. Therefore, $As^m\subseteq C$
and since $s$ is arbitrary, we conclude that $AI^m\subseteq C.$
 Thus, $BI=B$.
\end{proof}
\begin{corollary}({\bf Krull Intersection Theorem})\label{krull intersection theorem}
Let S be a Noetherian monoid in which its unique maximal right
ideal $\mathfrak{M} $ is two-sided. Let $A$ be a finitely generated
quasi-strongly faithful $S$-act with a unique zero element $\theta$.
Then $\cap_{n\in\mathbb{N}}AI^n=\{\theta\}$ for every proper ideal
$I$ of $S$.
\end{corollary}
\begin{proof}
Put $B=\cap_{n\in\mathbb{N}}AI^n$, by Theorem~\ref{krull} we have
$BI=B$. Since $B\leq A$ and $A$ is Noetherian and quasi-strongly
faithful then $B$ is finitely generated quasi-strongly faithful
$S$-act. Thus $B=\{\theta\}$ by Theorem~\ref{nakayama 2}.
\end{proof}
\begin{corollary}
Let S be a Noetherian monoid in which as an $S$-act is quasi-strongly
faithful and has a unique zero element $\theta$. Then:
\begin{enumerate}
    \item if $S$ is
commutative  then $\cap_{n\in\mathbb{N}}I^n=\{\theta\}$ for
every proper ideal $I$ of $S$.
    \item if the unique maximal right ideal $\mathfrak{M} $ of $S$ is two-sided then
$\cap_{n\in\mathbb{N}}\mathfrak{M}^n=\{\theta\}$.
\end{enumerate}
\end{corollary}
\begin{proof}
(1) and (2) are evident by Corollary~\ref{krull intersection theorem}.
\end{proof}

\centerline{\bf ACKNOWLEDGMENT}

The authors would like to thank Semnan university for its financial
support.

\end{document}